\documentclass[12pt,reqno,a4paper]{amsart}
\usepackage{amsmath,amssymb,amsthm,amsfonts,mathrsfs,amsopn}
\usepackage[all]{xy}
\usepackage{dsfont}
\usepackage{hyperref}
\usepackage{color}
\usepackage{esint}
\usepackage{mathtools}
\mathtoolsset{showonlyrefs}
\usepackage{slashed}

\usepackage{tikz-cd}

\usepackage[margin=1in]{geometry}

\usepackage[obeyDraft]{todonotes}

\newcommand{\parenthesis}[1]{\left(#1\right)} 
\newcommand{\braces}[1]{\left\{#1\right\}} 

\newcommand{\R}{\mathbb{R}}

\newcommand{\dd}{\mathop{}\!\mathrm{d}}
\newcommand{\p}{\partial}
\newcommand{\sph}{\mathbb{S}}

\newcommand{\average}[1]{\left<#1\right>_{_{\lambda}}} 
\newcommand{\oscil}[1]{\left[#1\right]_{_{\lambda}}} 

\DeclareMathOperator{\diag}{diag}

\DeclareMathOperator{\ds}{\dd{s}}
\DeclareMathOperator{\dx}{dx}

\DeclareMathOperator{\Ker}{Ker}
\DeclareMathOperator{\Mat}{Mat} 

\DeclareMathOperator{\rank}{rank}

\DeclareMathOperator{\Span}{Span}

\newtheorem{thm}{Theorem}[section]
\newtheorem{dfn}[thm]{Definition}
\newtheorem{lemma}[thm]{Lemma}
\newtheorem{prop}[thm]{Proposition}

\newtheorem{rmk}[thm]{Remark}

\newtheorem{theorem}{Theorem}

\def\ii{\infty}

\def\R{\mathbb{R}}
\def\m1{{I\!\!M}}

\renewcommand{\to}{\rightarrow}

\newcommand{\rife}[1]{(\ref{#1})}
\newcommand{\ov}[1]{\overline{#1}}

\newcommand{\intbar}{\mathop{\int\makebox(-15.5,0){\rule[6pt]{.7em}{0.3pt}}%
\kern-6pt}\nolimits}

\newcommand{\al}{\alpha}

\newcommand{\sg}{\sigma}

\newcommand{\om}{\Omega}
\newcommand{\lm}{\lambda}

\newtheorem{proposition}[theorem]{Proposition}
\newtheorem{corollary}[theorem]{Corollary}
\newtheorem{remark}[theorem]{Remark}
\newtheorem{definition}[theorem]{Definition}
\newcommand{\brm}{\begin{remark}\rm}
\newcommand{\erm}{\end{remark}}
\newcommand{\bdf}{\begin{definition}\rm}
\newcommand{\edf}{\end{definition}}
\newcommand{\bte}{\begin{theorem}}
\newcommand{\ete}{\end{theorem}}
\newcommand{\bpr}{\begin{proposition}}
\newcommand{\epr}{\end{proposition}}
\newcommand{\ble}{\begin{lemma}}
\newcommand{\ele}{\end{lemma}}
\newcommand{\bco}{\begin{corollary}}
\newcommand{\eco}{\end{corollary}}
\newcommand{\beq}{\begin{equation}}
\newcommand{\eeq}{\end{equation}}
\newcommand{\bdm}{\begin{displaymath}}
\newcommand{\edm}{\end{displaymath}}

\def\sideremark#1{\ifvmode\leavevmode\fi\vadjust{\vbox to0pt{\vss
 \hbox to 0pt{\hskip\hsize\hskip1em \vbox{\hsize2.1cm\tiny\raggedright\pretolerance10000 \noindent #1\hfill}\hss}\vbox to15pt{\vfil}\vss}}}

\begin{document}

\numberwithin{equation}{section}
\hfuzz=2pt
\frenchspacing

\title[]{A Courant nodal domain theorem for linearized mean field type equations}

\author[D.Bartolucci]{Daniele Bartolucci$^{\dag}$}
\address{Daniele Bartolucci, Department of Mathematics, University of Rome {\it "Tor Vergata"}, \\  Via della ricerca scientifica n.1, 00133 Roma, Italy. }
\email{bartoluc@mat.uniroma2.it}

\author[A. Jevnikar]{Aleks Jevnikar}
\address{Aleks Jevnikar, Department of Mathematics, Computer Science and Physics, University of Udine, Via delle Scienze 206, 33100 Udine, Italy.}
\email{aleks.jevnikar@uniud.it}

\author[R. Wu]{Ruijun Wu}
\address{Ruijun Wu, School of Mathematics and Statistics, Beijing Institute of Technology, Zhongguancun South Street No. 5, Haidian District, Beijing, P.R. China.}
\email{ruijun.wu@bit.edu.cn}



\thanks{$^{(\dag)}$Research partially supported by the MIUR Excellence Department Project\\ MatMod@TOV
awarded to the Department of Mathematics, University of Rome Tor Vergata.}

\begin{abstract}
We are concerned with the analysis of a mean field type equation and its linearization, which is a nonlocal operator, for which we estimate the number of nodal domains for the radial eigenfunctions and the related uniqueness properties.
\end{abstract}
\maketitle

\

\section{Introduction}
\setcounter{equation}{0}

Given a~$C^2$ function~$f\colon [0,+\infty)\to [0,\infty)$, satisfying $f'>0$ in $(0,+\ii)$
and for a fixed~$\lambda\ge 0$, on a smooth bounded domain~$\Omega\subset \R^n$,~$n\ge 2$,  we consider the constrained problem in the unknowns~$(\alpha,\psi)$:
\begin{align}\label{eq:MFE}
 \begin{cases}
  -\Delta \psi= f(\alpha+\lambda\psi), & \mbox{ in } \Omega, \\
  \int_{\Omega} f(\alpha+\lambda\psi)\dx =1, &  \\
   \alpha>0, & \\
  \psi>0 & \mbox{ in } \Omega, \\
  \psi=0, & \mbox{ on } \p\Omega.
 \end{cases}
\end{align}
For a fixed $\lm$, by definition a solution of~\eqref{eq:MFE} is a pair $(\alpha_\lambda,\psi_\lambda)$ where
$\psi_\lambda$ is a classical~$C^{2}(\ov{\om})$ solution of the elliptic equation. Let $(\alpha_\lambda,\psi_\lambda)$ be any such solution and set
\begin{align}
 V_{\lambda}
 = f'(\alpha_\lambda+\lambda\psi_\lambda) \in C^1(\Omega),
\end{align}
so that by our assumptions~$V_\lambda>0$ in~$\overline{\Omega}$.
In applications it also happens that~$V_\lambda>0$ in~$\Omega$ with~$V_\lambda$ vanishing on the boundary~$\p \Omega$, which will be particularly discussed in the concluding section.
Typical examples include $f(t)=e^t$ which yields to the well known mean field equations in dimension two,
see for example \cite{BGJM-MathAnn(2019),BJLY-ARMA(2018),BT-JDE(2019),ChenLinCPAM(2015)} and references quoted therein, as well as~$f(t)= t^p$ for some~$p\ge 1$ in general dimension, which is particularly relevant for the analysis of problems arising in plasma physics, see \cite{BJ1,BJ2} and references therein.

\

The linearized operator associated to~\eqref{eq:MFE} takes the form
\begin{align}
 L_\lambda(\phi)= -\Delta \phi-\lambda V_\lambda \oscil{\phi}
\end{align}
where
\begin{align}
 \oscil{\phi}=\phi-\average{\phi}, & \mbox{ with }
 \average{\phi}= \int_\Omega \frac{V_\lambda\phi}{\int_\Omega V_\lambda}.
\end{align}
The average term, which is a linear but non-local term, shows up due to the volume constraint in~\eqref{eq:MFE}.
Let~$\sigma$ be an eigenvalue of~$L_\lambda$, and~$\phi\in H^1_0(\Omega)\setminus \{0\}$ be an eigenfunction of~$\sigma$, that is by definition a weak solution of
\begin{align}\label{eq:eigenfunction}
 -\Delta\phi-\lambda V_\lambda \oscil{\phi}=\sigma V_\lambda\oscil{\phi}.
\end{align}
We will denote by $\sigma_{1,\lambda}$ the first eigenvalue of \rife{eq:eigenfunction} (see \rife{eq:first eigenvalue} below for definition).
Consider $f(t)=t^p$ and $\om$ a two-dimensional disk.
A natural question in this particular case arises from the results in \cite{BJ1,BJ2}, concerning a
problem in plasma physics, which asks whether or not
$\sigma_{1,\lambda}>0$ for any $\lm<\lm_*$, where $\lm_*$ is an explicit threshold depending only on the best constant of the Sobolev embedding $H^1_0(\om)\hookrightarrow L^{2p}(\om)$. This would imply, among other things, nice energy monotonicity properties which are the analogue of those arising in the context of classical mean field equations for $\lm<8\pi$, see \cite{B2,BartolucciJevnikar2021global} and references therein. This is our initial motivation to obtain refined information about the spectral properties of~$L_\lambda$, in particular for radial eigenfunctions on a disk.
However, as far as we know, some of the classical results at hand for "standard" eigenvalue problems, as for example the Courant nodal domain theorem \cite{Pleijel1956remarks}, and consequently neither the multiplicity of eigenfunctions \cite{Nadirashvili1987multiple}, are available so far about \rife{eq:eigenfunction}, as we discuss here after.

\bigskip

Integrating~\eqref{eq:eigenfunction} on the domain~$\Omega$ by parts gives
\begin{align}
 \int_{\p\Omega} \frac{\p\phi}{\p\nu}\ds = 0.
\end{align}
Thus, either~$\p_\nu \phi$ changes sign on~$\p\Omega$ or~$\p_\nu\phi\equiv 0$ on~$\p \Omega$.
In the former case, since~$\phi\in H^1_0(\Omega)$, we see that~$\phi$ also changes sign in~$\Omega$ and hence has at least two nodal domains. The observation which motivates part of this work is that in fact the latter case may also happen.

Note that the case~$\p_\nu\phi\equiv 0$ may only happen if~$\average{\phi}\neq 0$.
Indeed, as far as~$\average{\phi}=0$, the non-local character of~\eqref{eq:eigenfunction} drops out and the the classical Hopf lemma implies that~$\phi\equiv 0$ in~$\Omega$.
On the other side, if~$\average{\phi}\neq 0$, by the Hopf lemma we find that,
\begin{lemma}\label{lem1.1}
 Let~$\phi$ be an eigenfunction of~$\sigma$, i.e.~\eqref{eq:eigenfunction} holds.
 Let~$\Omega_1$ be a nodal domain satisfying an interior sphere condition at~$x_0\in \p \Omega_1$.
 If~$\phi<0$ in~$\Omega_1$ and~$\average{\phi}>0$, then
 \begin{align}
  \left.\frac{\p\phi}{\p \nu}\right|_{x_0} >0.
 \end{align}

\end{lemma}

\begin{proof}
 Since ~$\sigma+\lambda\ge 0$ (see \rife{eq:first eigenvalue}) and~$V_\lambda>0$, then the function~$\phi$ satisfies
 \begin{align}
  \Delta \phi + (\lambda+\sigma)V_\lambda \phi
  = (\lambda+\sigma)V_\lambda \average{\phi}\ge 0.
 \end{align}
 Then the classical Hopf lemma (see e.g. \cite{GilbargTrudinger2001elliptic}) applies at~$x_0$, immediately implying the claim.

\end{proof}

The assumption that~$\average{\phi}>0$ does not harm any generality: if~$\average{\phi}<0$, we consider~$\widetilde{\phi}=-\phi$ and conclude that if~$\phi>0$ in~$\Omega_1$ then
\begin{align}
 \left.\frac{\p\phi}{\p\nu}\right|_{x_0}<0.
\end{align}
But in the case~$\average{\phi}>0$ and~$\phi|_{\Omega_1}>0$, we cannot apply the Hopf lemma, and it can happen that $\left.\frac{\p\phi}{\p \nu}\right|_{x_0}=0.$
Actually, as mentioned above, it may even happen that~$\p_\nu\phi\equiv 0$ along~$\p\Omega_1$.
Indeed, as far as $\Omega$ is a disk~$B_r\subset \R^2$, this is verified for example in a special case (which however does not fit our assumptions since $V_{\lambda}\equiv 0$ in that situation) as discussed in \cite{Bartolucci2012,BartolucciJevnikar2021global} and more in general for a non-positive eigenvalue~$\sigma\le 0$.
The latter idea goes back to~\cite{LinNi1988counterexample}, where it was shown that any "standard" eigenfunction on a disk (that is any solution of \eqref{eq:eigenfunction} on a disk with $\average{\phi}=0$) whose eigenvalue $\sigma$ is non positive, must be radial. We postpone this proof to Section~\ref{sect:radial eigenfunctions}. Indeed we have,

\begin{lemma}\label{lemma:negative eigenvalue has radial eigenfunction}
 Let~$\phi\in H^1_0(B_1)$ be an eigenfunction of a non positive eigenvalue~$\sigma\le 0$.
 Then~$\phi$ is radial and~$\phi'(1)=0$.
\end{lemma}

In fact it is readily seen that if~$\phi$ is a radial eigenfunction, then, regardless of the sign of the eigenvalue, it satisfies~$\phi'(1)=0$.
In particular for the first eigenvalue in \cite{Bartolucci2012} and~\cite[Appendix]{BartolucciJevnikar2021global} (which is positive but $V_\lambda\equiv 0$ in that case)
there are three eigenfunctions, one of which being radial with~$\phi'(1)=0$ and the other two having two nodal domains.

This unusual phenomenon causes troubles with the theory of nodal domains. For example, on a general domain $\Omega$, a zero point of order greater or equal than two, that is, $x_0\in\ov{\om}$, $\phi(x_0)=0,\nabla \phi(x_0)=0$, need not be isolated as in classical linear problems \cite{Bers1956}. This is not surprising after
all, since, due to the non local term proportional to $\average{\phi}$, unlike standard linear growth problems \cite{Hartman-Wintner-AmJM(1953)}, near any such point we have that $|\Delta \phi|$ is not anymore controlled by $(|\phi|+|\nabla \phi|)$.

In this work, motivated also by the above mentioned plasma problem and by Lemma \ref{lemma:negative eigenvalue has radial eigenfunction}, we wish to make a first step
in this direction and consider radial eigenfunctions in the unit ball~$B_1\subset\R^n$.
The nodal sets will be spheres/solid shells and we will estimate the number of nodal domains.

Before that, let us briefly recall the Courant nodal domain theorem.
A~\emph{nodal domain} is any domain~$\Omega_0\subseteq \Omega$ such that~$\phi\equiv 0$ on~$\p\Omega_0$ and either~$\phi>0$ in~$\Omega_0$ or~$\phi<0$ in~$\Omega_0$.
The Courant nodal domain theorem~\cite{Pleijel1956remarks} says that any ~$n$-th eigenfunction (counted with multiplicity) has at most~$n$ nodal domains.
For the non-local operator~$L_\lambda$, it is expected that there is a similar bound for the number of nodal domains.
However, due to the volume constraint (which leads to the non-local term in the equation), any~$n$-th eigenfunction of~$L_\lambda$ in principle could be thought of as an~$(n+1)$-th eigenfunction of an unconstrained problem.
As a consequence in this case any~$n$-th eigenfunction of~$L_\lambda$ should have at most~$(n+1)$ nodal domains.
We will prove this fact for the first radial eigenfunction.
However, as far as~$\average{\phi}=0$, obviously the equation~\eqref{eq:eigenfunction} becomes a standard linear equation, whence the argument in~\cite{Pleijel1956remarks} works and gives

\begin{theorem}[]\label{plj}
Let~$\phi_k$ be a~$k$-th eigenfunction of~$L_\lambda$ with $k\geq 1$ and assume that~$\average{\phi}=0$.
 Then~$\phi_k$ has at most~$(k+1)$ nodal domains.
\end{theorem}

The nontrivial case is when~$\average{\phi}\neq 0$.
Consider the unit ball~$B_1\subset \R^n$ and let~$\phi=\phi(r)$ be a radial eigenfunction:
\begin{align}\label{eq:eigenfunction-radial}
 \phi''(r)+\frac{n-1}{r}\phi'(r)+\lambda V_\lambda \oscil{\phi}
 = -\sigma V_\lambda\oscil{\phi},\qquad r\in [0,1]
\end{align}
with~$\average{\phi}>0$.
Due to the above observations it may happen that~$\phi\ge 0$ in~$[r_1, r_3]$ and~$\phi(r_2)=0$,~$\phi'(r_2)=0$ for some~$r_2\in [r_1, r_3]$.
Then~\eqref{eq:eigenfunction-radial}  implies that
\begin{align}
 \phi''(r_2)=(\lambda+\sigma)V_\lambda(r_2)\average{\phi}>0.
\end{align}
Remark that since we assume $\al>0$ and $f'>0$ in $(0,+\ii)$, then $V_{\lambda}$ is a strictly positive even if $r_2=1$. In particular, any such point is necessarily isolated.
This fact motivates the following definitions.
\begin{dfn}\label{dfn:singular point-I}
 Let~$\phi$ be a radial eigenfunction of~\eqref{eq:eigenfunction-radial} with~$\average{\phi}>0$ in~$B_1$.
 A \emph{singular point} of~$\phi$ is a point~$r_0\in [0,1]$ such that
 \begin{align}
  \phi(r_0)=0, & &
  \phi'(r_0)=0, & &
  \phi''(r_0)>0.
 \end{align}

\end{dfn}

\begin{dfn}
 Let~$\phi$ be a radial eigenfunction of~$L_\lambda$ in~$B_1$ with~$\average{\phi}>0$.
 A \emph{generalized nodal domain} of~$\phi$ is a radial domain~$\Omega_0$ with the following properties:
 \begin{itemize}
  \item $\phi\equiv 0$ on~$\p\Omega_0$, and if $r\in [0,1)$ then $\p_r\phi \neq 0$ on~$\p \Omega_0$,
  \item in~$\Omega_0$, either~$\phi\ge 0$ or~$\phi\le 0$,
  \item if~$\phi\le 0$ in~$\Omega_0$ then~$\phi<0$ in~$\Omega_0$,
  \item if~$\phi\ge 0$ in~$\Omega_0$, then~$\phi>0$ in~$\Omega_0$ possibly with the exception of a finite number of spheres $\braces{x\in B_1 \mid |x|=r_i}_{i=1,2,\cdots, n}$ such that each~$r_i$ is a singular point of~$\phi(r)$, $1\le i\le n$.
 \end{itemize}

\end{dfn}

\begin{rmk}
 If a generalized domain is a ball~$B_r(0)$, then in polar coordinates we may identify~$B_r(0)$ with the interval~$[0,r)$, being understood that in this particular case the condition~$\phi=0$ on~$\p B_r(0)$ takes the form~$\phi(r)=0$.
\end{rmk}

Note also that, according to the above definition, the nodal sets, as the boundaries of the generalized nodal domains, are the preimages of some regular values.
Thus they are all \emph{nodal spheres/solid shells}. We can prove the following
\begin{thm}\label{thm:radial nodal domains-first}
 Let~$\phi_1$ be a radial first eigenfunction on~$B_1$.
 Then~$\phi_1$ has at most two generalized nodal domains.
\end{thm}

The proof is technically nontrivial, see Section~\ref{sect:proof of first thm}.
We refine the argument of~\cite{Pleijel1956remarks} and reduce the problem to that of finding at least one negative eigenvalue of a suitably defined matrix.
To achieve this goal we combine the~\emph{matrix determinant lemma} and the Sylvester criterion.
We also show the sharpness of the above result by some example.

For a general~$k$-th radial eigenvalue, the above argument cannot work directly.
Instead, we appeal to the \emph{Interlacing Theorem} for symmetric matrices, and get  finer information on the negative inertia index of certain coefficient matrices, which results in the following
 \begin{thm}\label{thm:radial nodal domains-general}
   Any~$k$-th radial function has at most~$2k$ generalized nodal domains.
 \end{thm}

 This is a generalization of Theorem~\ref{thm:radial nodal domains-first}.
 We remark that even if the proof of Theorem~\ref{thm:radial nodal domains-general} is elegant and self-contained, nevertheless we include the proof of Theorem~\ref{thm:radial nodal domains-first} since it uses a different strategy which is quite enlightening and shows the power of the Pleijel's original argument.
 Furthermore, there are two points to be clarified.
 First of all we are only enumerating the radial eigenfunctions, the non-radial ones are not included.
 The nonradial ones could help to fill the gap between~$(k+1)$ and~$2k$, although we don't have a precise argument at hand.
 Moreover, there is still a chance that, compared to the classical Courant Nodal Domain Theorem,
 the result is not sharp. Further comments about this point will be given at the end of Section~\ref{sect:proof of general case}.

In the classical case, it is well known that the maximum allowed number of nodal domains of eigenfunctions has relevant consequences about the multiplicity of the corresponding eigenvalues, see~\cite{Nadirashvili1987multiple}.
Things are different for \eqref{eq:eigenfunction-radial} in the radial case. Indeed we have,
\begin{prop}\label{prop:uniqueness}
 Any eigenvalue~$\sigma$ of~$L_\lambda$ in~$B_1$ has at most one radial eigenfunction.
\end{prop}
Needless to say that there may be no radial eigenfunctions for a positive eigenvalue.
But if there is one, the above proposition claims that this is the only one.
The proof is given in Section~\ref{sect:uniqueness}.

Let us remark that, besides Theorem A, by a result in \cite{Damascelli1999631} we see that there is no eigenfunction $\phi$ of \eqref{eq:eigenfunction-radial} with $\average{\phi}=0$. Therefore we deduce from Proposition \ref{prop:uniqueness} and Theorem \ref{thm:radial nodal domains-first} that, the unique  eigenfunction of \eqref{eq:eigenfunction-radial} of a non positive first eigenvalue admits, as discussed above, at most two (generalized) nodal domains.

\

Several questions remain open. For example in many applications one would need some sort of generalized nodal domain theorem, limiting the number of generalized nodal domains for the eigenfunctions of~\eqref{eq:eigenfunction}, both for radial higher eigenfunctions with a sharper bound as well as for any such eigenfunction (not
necessarily radially symmetric) in general domains. Then one would like to understand also the
multiplicity (\cite{Nadirashvili1987multiple}) for problems of this sort.

In conclusion we observe that the assumptions $f'>0$ in $(0,+\infty)$ and $\al>0$ are still too restrictive
to cover the problem arising in plasma physics (\cite{BJ1,BJ2}) which was indeed part of our initial motivation.
To achieve this goal we need a refined version of Lemma~\ref{lemma:negative eigenvalue has radial eigenfunction}, Theorem~\ref{thm:radial nodal domains-first} and Proposition~\ref{prop:uniqueness} under an additional
technical assumption about $V_\lambda$ (see \eqref{eq:model assumption} below) in case $\al=0$. To simplify the exposition we postpone the discussion concerning this technical point to Section~\ref{sec6}.

\

The paper is organized as follows. In section \ref{sect:basic} we collect some preliminary spectral properties, then, we prove the main nodal domain theorem and the related multiplicity of eigenvalues in sections \ref{sect:proof of first thm} and \ref{sect:uniqueness}, respectively. The radial eigenfunctions with non positive eigenvalues are discussed in \ref{sect:radial eigenfunctions}. The last section \ref{sec6} is devoted to a degenerate case arising in the plasma problem.

\

\noindent {\bf Acknowledgments.}
We would like to express our warmest thanks to Prof. Carmine Di Fiore for very interesting discussions and in particular for pointing out to us the relevance of the Courant-Fischer Interlacing Theorem.

\

\section{Basic spectral properties} \label{sect:basic}

The first eigenvalues of~$L_\lambda$ can be characterized by the min-max principle:
\begin{align}\label{eq:first eigenvalue}
 \sigma_1=\sigma_1(\alpha_\lambda,\psi_\lambda)
 = \min_{\phi\in H^1_0(\Omega)}
 \frac{ \int_\Omega |\nabla\phi|^2-\lambda\int_\Omega V_\lambda \oscil{\phi}^2  }{\int_\Omega V_\lambda \oscil{\phi}^2},
\end{align}
and for~$k\ge 2$, the~$k$-th eigenvalues are defined inductively by
\begin{align}\label{eq:general eigenvalue}
 \sigma_k=\sigma_k(\alpha_\lambda,\psi_\lambda)
 = \min_{\phi\in H^1_0(\Omega), \; \average{\phi\phi_j}=0, \forall 1\le j\le k-1 }
 \frac{ \int_\Omega |\nabla\phi|^2-\lambda\int_\Omega V_\lambda \oscil{\phi}^2  }{\int_\Omega V_\lambda \oscil{\phi}^2},
\end{align}
where~$\phi_j$ (counted with multiplicity) is any eigenfunction of the~$j$-th eigenvalue~$\sigma_j$, for~$j=1,\cdots, k-1$.
In particular,~$\lambda+\sigma>0$ for any eigenvalue~$\sigma$ and for any~$\phi\in H^1_0(\Omega)$,
\begin{align}\label{eq:lower bound by sigma1}
 \int_\Omega |\nabla\phi|^2 -\lambda\int_\Omega V_\lambda \oscil{\phi}^2 \ge \sigma_1 \int_\Omega V_\lambda \oscil{\phi}^2
\end{align}
and the equality is attained for the eigenfunctions of~$\sigma_1$.

This should be compared with the first Dirichlet type eigenvalues, by which we mean
\begin{align}
 \nu_1
 = \nu_1(\alpha_\lambda,\psi_\lambda)
 =\min_{\phi\in H^1_0(\Omega)}
 \frac{ \int_\Omega |\nabla\phi|^2-\lambda\int_\Omega V_\lambda {\phi}^2  }{\int_\Omega V_\lambda {\phi}^2},
\end{align}
which is evaluated without taking off the average~$\average{\phi}$.
Indeed, since~$V_\lambda>0$ in~$\Omega$, we have
\begin{align}
 \sigma_1+ \lambda
 = \min_{\phi\in H^1_0(\Omega)\setminus\braces{0}}
 \frac{ \int_\Omega |\nabla\phi|^2 }{\int_\Omega V_\lambda \oscil{\phi}^2}
 \ge \min_{\phi\in H^1_0(\Omega)\setminus\braces{0}}
 \frac{ \int_\Omega |\nabla\phi|^2 }{\int_\Omega V_\lambda {\phi}^2}
 =\nu_1+\lambda
\end{align}
where we used that, for any~$\phi\in H^1_0(\Omega)\setminus\braces{0}$, $\int_{\Omega} V_\lambda \oscil{\phi}\dx =0$ and consequently
\begin{align}
 \int_{\Omega} V_\lambda \phi^2\dx
 =&\int_\Omega V_\lambda (\average{\phi}+\oscil{\phi})^2\dx \\
 =&\int_\Omega V_\lambda \oscil{\phi}^2\dx + \average{\phi}^2 \int_\Omega V_\lambda \dx
 \ge \int_\Omega V_\lambda \oscil{\phi}^2\dx.
\end{align}
The equality above holds iff~$\average{\phi}=0$.
Concerning the eigenvalues, we readily deduce that~$\sigma_1 \ge \nu_1$, although the equality cannot hold, since the first Dirichlet eigenfunction has a fixed sign in~$\Omega$, whence it cannot have zero mean with respect to~$V_\lambda$.
Therefore~$\sigma_1> \nu_1$.

\bigskip

\section{A nodal domain theorem for first radial eigenfunctions}
\label{sect:proof of first thm}

In this section we carry out the proof of Theorem~\ref{thm:radial nodal domains-first}.
Here and in the sequel we assume without loss of generality that~$m_\lambda\equiv \average{\phi_1}>0$, otherwise the result is well known.\\
First of all observe that the zeros of a radial eigenfunction~$\phi(r)$, viewed as a function on the closed unit interval, are isolated, regardless of being singular or not.
Actually if there were infinitely many zeros, then these radii would admit an accumulation point $r_0\in[0,1]$ at which~$\phi^{''}(r_0)=\phi^{'}(r_0)=\phi(r_0)=0$. This impossible under our assumptions since~$V_\lambda>0$ on~$\bar{B}_1$ and we would deduce from \eqref{eq:eigenfunction-radial} that
$$
 \phi''(r_0)= (\lambda+\sigma) V_\lambda(r_0)\average{\phi}\neq 0.
$$
Therefore $\phi^{-1}(0)\subset [0,1]$ is a finite set and
there are at most finitely many generalized nodal domains,
which are concentric annuli, or more precisely, solid shells.

For any radial function solving~\eqref{eq:eigenfunction}, integrating by parts over~$\Omega\equiv B_1$ we have,
\begin{align}
 0=\int_{\p B_1} \p_\nu \phi_1\ds = 2\pi \phi'_1(1)
\end{align}
whence~$\phi'_1(1)=0$. In particular~$r=1$ is an isolated singular point, since
$\phi''(1)= (\lambda+\sigma) V_\lambda(1)\average{\phi}\neq 0.$\\

We argue by contradiction and assume that~$\phi_1$ has~$N$ generalized nodal domains, for some~$N\ge 3$.
As~$\phi'_1(1)=0$ and~$\phi_1\in H^1_0(B_1)$, we deduce from Lemma \ref{lem1.1} that~$\phi_1$ is nonnegative in the outer-most generalized nodal domain, which is denoted by~$\Omega_1$; so there is some~$r_1<1$ such that
\begin{align}
 \Omega_1 =\braces{x\in B_1\mid r_1< r < 1}.
\end{align}
Then~$\Omega_2=\braces{x\in B_1\mid  r_2<|x|<r_1}$ for some~$r_2\in (0,r_1)$, in which~$\phi_1<0$.
In this way we see that in the shells~$\Omega_{2k+1}$ the eigenfunction~$\phi_1$ is nonnegative (for~$2k+1\le N$), while in the shells~$\Omega_{2k}$ it is negative, as long as~$2k\le N$.
Note that the inner-most generalized nodal domain is a ball.

For each~$j=1,2,\cdots, N$, let
\begin{align}
 \phi_{1,j}\coloneqq \phi_1 \cdot \chi_{\Omega_j}
\end{align}
where~$\chi_{\Omega_j}$ stands for the characteristic function for~$\Omega_j$.
Then~$\phi_{1,j}\in H^1_0(\Omega)$ with weighted average
\begin{align}
 m_j\equiv \average{\phi_{1,j}}
 =\frac{\int_{\Omega_j}\phi_1  V_\lambda\dx}{\int_{\Omega} V_\lambda\dx}.
\end{align}
Then~$m_1>0$,~$m_2<0$,~$m_3>0$, etc. and in general~$(-1)^j m_j <0$. Moreover,
\begin{align}
 m_1+m_2 + \cdots + m_N =\average{\phi_1}=m_\lambda >0.
\end{align}
Consider the test function
\begin{align}
 \varphi= \sum_{j=1}^N a_j\phi_{1,j}  \in H^1_0(\Omega)
\end{align}
for some~$(a_1, a_2,\cdots, a_N)\in \R^N$ to be fixed later on.
The weighted average of~$\varphi$ is
\begin{align}
 \average{\varphi}=\sum_{j=1}^N a_j m_j.
\end{align}

\

By integrating by parts we find that,
\begin{align}
 &\frac{1}{\lambda+\sigma_1}\int_\Omega |\nabla\varphi|^2\dx \\
 =& \frac{1}{\lambda+\sigma_1} \int_{\Omega}\varphi(-\Delta\varphi)\dx
   =\sum_{j=1}^3\frac{1}{\lambda+\sigma_1}\int_{\Omega_j} \varphi(-\Delta\varphi)\dx  \\
 =& \sum_{j=1}^N a_j \int_{\Omega_j} \varphi V_\lambda \oscil{\phi_1}\dx \\
 =& \sum_{j=1}^N a_j \int_{\Omega_j} \varphi V_\lambda \parenthesis{\phi_1 - \average{\phi_1}} \dx  \\
 =& \sum_{j=1}^N \int_{\Omega} \varphi V_\lambda a_j\phi_{1,j}\dx
   -\average{\phi_1}\sum_{j=1}^N a_j \int_{\Omega_j} V_\lambda\varphi\dx \\
 =& \int_{\Omega} \varphi V_\lambda \varphi \dx
   -\average{\phi_1} \parenthesis{\int_\Omega V_\lambda\dx}\sum_{j=1}^N m_j a_j^2  \\
 =& \int_\Omega V_\lambda \parenthesis{\varphi^2 -\average{\varphi}^2}\dx
   +\parenthesis{\int_\Omega V_\lambda\dx} \average{\varphi}^2
   -\parenthesis{\int_\Omega V_\lambda\dx} \average{\phi_1}
   \parenthesis{\sum_{j=1}^N m_j a_j^2}  \\
 =& \int_\Omega V_\lambda \oscil{\varphi}^2\dx
   +\parenthesis{\int_\Omega V_\lambda\dx}
   \braces{
     \parenthesis{\sum_{j=1}^N a_j m_j}^2
     -\parenthesis{\sum_{i=1}^N m_i} \parenthesis{\sum_{j=1}^N m_j a_j^2}
   }.
\end{align}
According to~\eqref{eq:lower bound by sigma1}, the tail term above is non-negative, namely the quadratic form
\begin{align}
 Q(\vec{a})=\sum_{i,j}m_i m_j a_i a_j - m_\lambda \sum_{j=1}^N m_j a_j^2
\end{align}
in~$\vec{a}=(a_1, \cdots, a_N)\in \R^N$ should be non-negative.
Equivalently, the symmetric matrix~$A$ corresponding to the quadratic form~$Q$, as given by,
\begin{align}
 A = -m_\lambda \begin{pmatrix}  m_1 & & & \\ & m_2 & & \\ & & \ddots &  \\ & & & m_N \end{pmatrix}
 +\begin{pmatrix}
  m_1^2 & m_1 m_2 & \ldots & m_1 m_N \\
  m_2 m_1 & m_2^2 & \ldots & m_2 m_N \\
  \vdots & \vdots & \ddots & \vdots \\
  m_N m_1 & m_N m_2 & \ldots& m_N^2
  \end{pmatrix}
\end{align}
doesn't have negative eigenvalues.
Note that~$A$ has a kernel given by
\begin{align}
 \Span_\R\braces{(1,1,\cdots, 1)}
\end{align}
which corresponds to~$\Span{\R\braces{\phi_1}}$: it is clear that this one-dimensional space lies in the kernel, and we
will show in a later section that this is indeed the full kernel.
Alternatively, one can prove that~$\rank(A)=N-1$ by an elementary computation.

In the rest of the proof we will show that if~$N\ge 3$ then~$A$ would have negative eigenvalues, in contradiction with~\eqref{eq:lower bound by sigma1}.
As a consequence we deduce that~$\phi_1$ cannot have more than two generalized nodal domains.

Observe at first that for~$N=2$ the matrix~$A\in \Mat(2\times 2;\R)$ has eigenvalues
\begin{align}
 t_0=0, & & t_1=-2m_1 m_2 >0,
\end{align}
while for~$N=3$ the matrix~$A\in \Mat(3\times 3; \R)$ has eigenvalues~$\braces{t_0, t_1, t_2}$ satisfying
\begin{align}
 t_0=0, & &  t_1 t_2 = 3m_1 m_2 m_3 m_\lambda <0.
\end{align}
Hence~$t_1$ and~$t_2$ are nonzero and have different signs.
In particular,~$A$ has a negative eigenvalue, a contradiction.

This actually proves that~$\phi_1$ cannot have three generalized nodal domains.
Next we use the same idea to prove the general case.
We remark that the case~$N>3$ cannot be directly reduced to the case~$N=3$ (as in~\cite{Pleijel1956remarks}), since there is a nonlocal term~$\average{\phi}$ in the equation and hence in~$A$.
The proof turns out to be more involved. In the sequel we assume that~$N\ge 4$.

Since any~$m_j$ is not zero, we can transform to the new variables
\begin{align}
 b_j=m_j a_j, \qquad j=1,\cdots, N.
\end{align}
In terms of~$b_j$'s the quadratic form~$Q$ takes the form
\begin{align}
 Q(\vec{a})
 = &\sum_{i,j=1}^N m_i m_j a_i a_j - m_\lambda \sum_{j=1}^N m_j a_j^2 \\
 =&\sum_{i,j=1}^N b_i b_j -m_\lambda \sum_{j=1}^N \frac{1}{m_j} b_j^2
\end{align}
and the matrix~$A$ transforms into
\begin{align}
 B= -m_\lambda \begin{pmatrix}  \frac{1}{m_1} & & & \\ & \frac{1}{m_2} & & \\ & & \ddots &  \\ & & & \frac{1}{m_N} \end{pmatrix}
 +\begin{pmatrix}
  1 & 1 & \ldots & 1 \\
  1 & 1 & \ldots & 1 \\
  \vdots & \vdots & \ddots & \vdots \\
 1 & 1 & \ldots& 1
 \end{pmatrix}
\end{align}
whose kernel is now given by the span of the vector~$(m_1, \cdots, m_N)$.
Moreover,~$B$ and~$A$ has the same eigenvalues.

To get the spectral properties of~$B$, one should look at its restriction onto~$\Ker(B)^\bot$.
However, in that orthogonal subspace, we didn't find an easy way to handle~$B$.
Instead we consider a complement of~$\Ker(B)$ given by~$\R^{N-1}=(0,\cdots, 0, 1)^\bot$, on which the matrix~$B$ takes the form
\begin{align}\label{mat:B restricted}
 B_{N-1}\coloneqq -m_\lambda \begin{pmatrix}  \frac{1}{m_1} & & & \\ & \frac{1}{m_2} & & \\ & & \ddots &  \\ & & & \frac{1}{m_{N-1}} \end{pmatrix}
 +\begin{pmatrix}
  1 & 1 & \ldots & 1 \\
  1 & 1 & \ldots & 1 \\
  \vdots & \vdots & \ddots & \vdots \\
 1 & 1 & \ldots& 1
 \end{pmatrix}
 \in \Mat((N-1)\times (N-1);\R).
\end{align}
Note that the second summand is a matrix of rank one.
If we write~$\vec{1}=(1,\cdots, 1)^T\in \R^{N-1}$, then
\begin{align}
 B_{N-1} = -m_\lambda \diag(\frac{1}{m_1},\cdots, \frac{1}{m_{N-1}})
 + \vec{1}\otimes \vec{1}
 \equiv H_0 + H_1,
\end{align}
where~$H_0$ denotes the diagonal part and~$H_1$ denotes the rank-one part.
By the matrix determinant lemma we have
\begin{align}
 \det(B_{N-1})
 =&\det(H_0 + H_1)  \\
 =& \parenthesis{1+\vec{1}^T (H_0)^{-1} \vec{1}} \det(H_0) \\
 =& \parenthesis{1-\sum_{i=1}^{N-1}\frac{m_i}{m_\lambda}}\cdot
 \frac{(-m_\lambda)^{N-1}}{m_1 m_2 \cdots m_{N-1}} \\
 =& (-1)^{N-1} \frac{m_\lambda^{N-2} m_N}{m_1\cdots m_{N-1}},
\end{align}
and we readily deduce in particular that~$B_{N-1}$ has no vanishing eigenvalues.

At this point we claim that~$B_{N-1}$ has at least one negative eigenvalue.
Argue by contradiction and assume that~$B_{N-1}$ is positive definite.
By the Sylvester criterion, the leading principal minors of~$B_{N-1}$ must all have positive determinant.
For~$j=1,2,\cdots, N-1$, let~$B_{N-1}^{(j)}$ denote the upper left~$(j\times j)$ corner, whose determinant is the~$j$-th leading principal minor, then again using the matrix determinant lemma:
\begin{align}\label{eq:B:N-1}
 \det B_{N-1}^{(N-1)}
 =& (-1)^{N-1} \frac{m_\lambda^{N-2}}{m_1 m_2\cdots m_{N-1}} m_N,
\end{align}
\begin{align}\label{eq:B:N-2}
 \det B_{N-1}^{(N-2)}
 =& (-1)^{N-2}\frac{m_\lambda^{N-3}}{m_1 m_2\cdots m_{N-2}} \parenthesis{m_N + m_{N-1}},
\end{align}
\begin{align}\label{eq:B:N-3}
 \det B_{N-1}^{(N-3)}
 =& (-1)^{N-3}\frac{m_\lambda^{N-4}}{m_1 m_2\cdots m_{N-3}} \parenthesis{m_N + m_{N-1}+ m_{N-2}},
\end{align}
\begin{align}\label{eq:B:N-4}
 \det B_{N-1}^{(N-4)}
 =& (-1)^{N-4}\frac{m_\lambda^{N-5}}{m_1 m_2\cdots m_{N-4}} \parenthesis{m_N + m_{N-1}+ m_{N-2}+ m_{N-3}}.
\end{align}

Since~$B_{N-1}$ was assumed to be positive definite, they should all be positive.

Recall that
\begin{align}
 m_1>0, & & m_2<0, & & \cdots\cdots  & & (-1)^N m_N <0,
\end{align}
and
\begin{align}
 m_\lambda = m_1 + m_2 + \cdots + m_N >0.
\end{align}

\

\noindent\textbf{Case 1: $N\equiv 0\pmod{4}$.}

In this case, there is an even number of negative~$m_j$'s so that
\begin{align}
 m_1 m_2 \cdots m_N >0.
\end{align}
Then by~\eqref{eq:B:N-1}
\begin{align}
 \det B_{N-1}^{(N-1)} <0
\end{align}
which is a contradiction.

\

\noindent\textbf{Case 2: $N\equiv 3\pmod{4}$.}
In this case there is an odd number of negative~$m_j$'s so that
\begin{align}
 m_1 m_2 \cdots m_N <0.
\end{align}
Hence by~\eqref{eq:B:N-1}
\begin{align}
 \det B_{N-1}^{(N-1)} <0
\end{align}
which is again a contradiction.

\

\noindent\textbf{Case 3: $N\equiv 1\pmod{4}$.}

Counting the negative signs in the sequence~$(m_j)$ we find that
\begin{align}
 m_1 m_2 \cdots m_{N-2} <0, & &
 m_1 m_2 \cdots m_{N-3} <0.
\end{align}
Since both~\eqref{eq:B:N-2} and~\eqref{eq:B:N-3} are assumed to be positive, we have
\begin{align}
 m_N+ m_{N-1}>0, & & m_N + m_{N-1} + m_{N-2} <0.
\end{align}
But the above cannot hold simultaneously since~$m_{N-2}= m_{4k-1}>0$ where~$N=4k+1$.

\

\noindent\textbf{Case 4: $N\equiv 2\pmod{4}$.}
Similarly, counting the negative signs of the~$m_j$'s we find that
\begin{align}
 m_1 m_2 \cdots m_{N-2} >0, & &
 m_1 m_2 \cdots m_{N-3} <0.
\end{align}
From~\eqref{eq:B:N-3} and~\eqref{eq:B:N-4} we would conclude
\begin{align}
 m_N+ m_{N-1} + m_{N-2}>0, & &
 m_N + m_{N-1} + m_{N-2}+ m_{N-3} <0.
\end{align}
Since for~$N=4k+2$, ~$m_{N-3}=m_{4k-1}>0$, the above two inequality cannot hold simultaneously.

To summarize, the restricted matrix~$B_{N-1}$ cannot be positive definite.
Hence the original matrix~$B$, as well as~$A$, must have negative eigenvalues.
This, as remarked, contradicts the min-max principle~\eqref{eq:lower bound by sigma1}.
Therefore, the radial first eigenfunction~$\phi_1$ cannot have more than two generalized nodal domains,
as claimed.


\

\section{A nodal domain theorem for general radial eigenfunctions}
\label{sect:proof of general case}

We start by recalling the \emph{Interlacing theorem} which is a consequence of the well-known Courant-Fischer min-max principle, see e.g.~\cite[Chapter 8]{Golub2013matrix} and the references therein.

\begin{theorem}[Interlacing Theorem]
 Let~$K_0$ be a symmetric~$N\times N$ matrix, and~$K_1= v^T \otimes v$ a rank-one matrix generated by a column vector~$v\in \R^n$.
 Then for~$1\le j\le N-2$
 \begin{align}
  \lambda_j(K_0 + K_1) \le \lambda_{j+1}(K_0) \le \lambda_{j+2}(K_0 + K_1), \\
  \lambda_j(K_0 ) \le \lambda_{j+1}(K_0+ K_1) \le \lambda_{j+2}(K_0).
 \end{align}

\end{theorem}

We apply this theorem to the matrix~$A=K_0 + K_1$, with
\begin{align}
 K_0 = -m_\lambda \begin{pmatrix}  m_1 & & & \\ & m_2 & & \\ & & \ddots &  \\ & & & m_N \end{pmatrix},
\quad
 K_1=\begin{pmatrix}
  m_1^2 & m_1 m_2 & \ldots & m_1 m_N \\
  m_2 m_1 & m_2^2 & \ldots & m_2 m_N \\
  \vdots & \vdots & \ddots & \vdots \\
  m_N m_1 & m_N m_2 & \ldots& m_N^2
  \end{pmatrix}
  =\mathbf{m}^T \mathbf{m}
\end{align}
where~$\mathbf{m}^T=(m_1, m_2,\cdots, m_N)\in \R^N$.
Then we get
\begin{align}
 \lambda_j(A) \le \lambda_{j+1}(K_0).
\end{align}
By the conditions on~$m_i$'s,~$K_0$ has precisely~$N_*\equiv \lceil \frac{N}{2} \rceil$ negative eigenfunctions.
Therefore,
\begin{align}
 \lambda_{N_*-1}(A) \le \lambda_{N_*}(K_0)<0 < \lambda_{N_* +1}(K_0) \le \lambda_{N_* +2}(A),
\end{align}
with only the signs of~$\lambda_{N_*}(A)$ and~$\lambda_{N_* +1}(A)$ left undetermined--but we know that one of them has to be zero!
The min-max principle tells that for the~$k$-the eigenfunction,
\begin{align}
 N_*-1 \le k-1
\end{align}
which implies~$N\le 2k$, namely the~$k$-th eigenfunction has at most~$2k$ generalized nodal domains.

This result is sharp in view of the first radial eigenfunction, as we have seen in the previous section.
For higher radial eigenfunctions, we try to show the sharpness from the matrix viewpoint by some examples.\\
For the second radial eigenfunction, i.e.~$k=2$, consider the matrix~$A$ with
 \begin{align}
  m_1=+5, & & m_2=-3, & & m_3=+5, & & m_4=-3,
 \end{align}
 so that~$m_\lambda=+4$.
 The corresponding matrix~$A$ has eigenvalues
 \begin{align}
  60, \quad 12, \quad 0, \quad -20,
 \end{align}
hence precisely one negative eigenvalue.
The enlarged matrix with~$m_j$ as above,~$1\le j\le 4$, while~$m_5=+5$, would have two negative eigenvalues:
\begin{align}
 75, & & 27, & & 0, & & -45, & & -45.
\end{align}
This cannot happen in terms of the min-max principle~\eqref{eq:general eigenvalue}.

For the third radial eigenfunction, i.e.~$k=3$, similarly consider a matrix with
\begin{align}
  m_1=+5, & & m_2=-3, & & m_3=+5, & & m_4=-3, & & m_5=+5, & & m_6=-3,
 \end{align}
 so that~$m_\lambda=+6$.
 The corresponding matrix~$A$ has eigenvalues
 \begin{align}
  90, & & 18, & & 18, & & 0, & & -30, & & -30,
 \end{align}
 with negative inertia index~$2$!
 If we increase the size~$N$ to~$7$, with~$m_7=+5$, then
 the corresponding~$A_{7\times 7}$ has eigenvalues
 \begin{align}
  105, & & 33, && 33,& & 0, & & -55, & & -55, & & -55,
 \end{align}
 which would again contradict the min-max principle~\eqref{eq:general eigenvalue}.

As remarked in the introduction, we don't know whether the bound~$2k$ is sharp among all radial eigenfunctions.
Note that we cannot, in general, hope for a linear bound of the form~$k+a$, since there exist non-radial eigenfunctions even on a radially symmetric domain such as the ball.
It thus remains open to find the optimal bound of the number of nodal domains.

\

\section{On eigenfunctions with non positive eigenvalues} \label{sect:radial eigenfunctions}

We present here the proof of Lemma~\ref{lemma:negative eigenvalue has radial eigenfunction}.  Let~$e_k(\theta)$,~$\theta\in\sph^{n-1}$, denote the eigenfunctions of the Laplace operator on~$\sph^1$ for the eigenvalues
 \begin{align}
  0=\mu_1< (n-1)=\mu_2 \le \mu_3 \le \cdots,
 \end{align}
 In particular,~$\int_{\sph^{n-1}} e_k(\theta)\dd\theta=0$ for~$k\ge 2$.

 Consider the eigenfunction~$\phi$:
 \begin{align}
  -\Delta\phi-\lambda V_\lambda \oscil{\phi}=\sigma V_\lambda\oscil{\phi}, \qquad \mbox{ in } B_1.
 \end{align}
Since~$\psi_\lambda$ is radially decreasing thanks to~\cite{GidasNiNirenberg1979symmetry}, then~$V_\lambda= f'(\alpha_\lambda+\lambda\psi_\lambda)\geq 0$.
 For~$k\ge 1$, consider the functions~$\bar{\phi}^k\colon (0,1]\to \R$ defined by
 \begin{align}
  \bar{\phi}^k(r)\coloneqq \int_{\sph^{n-1}}^{2\pi} \phi(r,\theta) e_k(\theta)\dd\theta.
 \end{align}
 Then we would have~$\phi=\sum_{k\ge 0} \bar{\phi}^k(r) e_k(\theta)$.
 Note that,~$e_1=1$ is constant and hence~$\bar{\phi}^1 e_1(\theta)$ is a radial function.

 We claim that~$\bar{\phi}^k(r)\equiv 0$ for all~$k\ge 2$.
 By the boundary condition we know that~$\bar{\phi}^k(1)=0$.
 In the interval~$(0,1)$, the function~$\bar{\phi}^k$ satisfies the ODE
 \begin{align}\label{eq:phi-k-radial ODE}
 \frac{\dd^2}{\dd r^2}\bar{\phi}^k
 +\frac{n-1}{r}\frac{\dd}{\dd r}\bar{\phi}^k
 -\frac{\mu_k}{r^2}\bar{\phi}^k
 +\lambda V_\lambda(r) \bar{\phi}^k
 =-\sigma V_\lambda(r)\bar{\phi}^k,
 \end{align}
 as the average part~$\average{\phi}$ doesn't contribute in the integration with respect to~$e_k(\theta)\dd\theta$.
 Suppose~$\bar{\phi}^k$ is not identically zero and let~$r_0$ be the first zero of~$\bar{\phi}^k$.
 W.l.o.g. we may assume that~$\bar{\phi}^k>0$ in~$(0,r_0)$.
 Note that the radial solution~$\psi_\lambda$ satisfies
 \begin{align}\label{eq:ODE for radial psi}
  \frac{\dd^2}{\dd r^2} (\psi'_\lambda)
  +\frac{n-1}{r}\frac{\dd}{\dd r}\psi'_{\lambda}
  +\parenthesis{\lambda V_\lambda -\frac{n-1}{r^2}}\psi'_\lambda=0
 \qquad \mbox{ in } (0,1)
 \end{align}
 and~$\psi'_\lambda(0)=0$, ~$\psi'_\lambda(r)\le 0$ for~$r\in (0,1]$.
 Therefore, multiplying both sides of~\eqref{eq:phi-k-radial ODE} by~$r^{n-1}\psi'_\lambda(r)$ and integrating over~$(0,r_0)$, we have,
 \begin{align}
  \int_0^{r_0} r^{n-1}\psi'_\lambda\frac{\dd^2}{\dd r^2}\bar{\phi}^k
 +(n-1)r^{n-2}\psi'_\lambda\frac{\dd}{\dd r}\bar{\phi}^k
 -\mu_k r^{n-3}\psi'_\lambda \bar{\phi}^k
 +\lambda r^{n-1} V_\lambda(r)\psi'_\lambda \bar{\phi}^k \dd r\\
 =-\sigma \int_0^{r_0} r^{n-1} V_\lambda(r)\psi'_\lambda\bar{\phi}^k \dd r.
 \end{align}
Integration by parts gives
\begin{align}
 \int_0^{r_0} r^{n-1}\psi'_\lambda\frac{\dd^2}{\dd r^2}\bar{\phi}^k \dd r
 &=r_0^{n-1} \psi'_\lambda(r_0)\frac{\dd \bar{\phi}^k}{\dd r}(r_0)\\
  +&\int_0^{r_0} \parenthesis{(n-1)(n-2)r^{n-3}\psi'_\lambda  +2(n-1)r^{n-2}\psi''_\lambda + r^{n-1}\psi'''_\lambda } \bar{\phi}^k \dd r,
\end{align}
\begin{align}
 \int_0^{r_0} (n-1)r^{n-2}\psi'_\lambda\frac{\dd}{\dd r}\bar{\phi}^k \dd r
 =\int_0^{r_0} -(n-1)(n-2)r^{n-3}\psi'_\lambda \bar{\phi}^k -(n-1)r^{n-2}\psi''_\lambda\bar{\phi}^k\dd r,
\end{align}
where we have used the boundary conditions~$\psi'_\lambda(0)=0$,~$\bar{\phi}^k(r_0)=0$.
Thus,
\begin{multline}
 r_0^{n-1}\psi'_\lambda(r_0)\frac{\dd\bar{\phi}^k}{\dd r}(r_0)
 +\int_0^{r_0}  r^{n-1} \parenthesis{\frac{\dd^2}{\dd r^2}\psi'_\lambda+\frac{n-1}{r}\frac{\dd}{\dd r}\psi'_\lambda+\lambda V_\lambda \psi'_\lambda }\bar{\phi}^k\dd r \\
 -\int_0^{r_0} \mu_k r^{n-3}\psi'_\lambda\bar{\phi}_k \dd r
 =-\sigma \int_0^{r_0} r^{n-1} V_\lambda(r) \psi'_\lambda\bar{\phi}^k\dd r.
\end{multline}
Then~\eqref{eq:ODE for radial psi} implies that,
\begin{align}
 r_0^{n-1} \psi'_\lambda(r_0)\frac{\dd\bar{\phi}^k}{\dd r}(r_0)
 +\int_0^{r_0} r^{n-3}(n-1-\mu_k)V_\lambda \psi'_\lambda \bar{\phi}^k\dd r
 =-\sigma\int_0^{r_0} r^{n-1} V_\lambda \psi'_\lambda \bar{\phi}^k\dd r.
\end{align}

Note that~$\mu_k\ge n-1$ for~$k\ge 2$, and~$\frac{\dd}{\dd r}\bar{\phi}^k(r_0)< 0$.
Thus the ~$l.h.s.$ of this equality is positive.
On the other side, if~$\sigma\le 0$, then the~$r.h.s.$ is non-positive unless~$\bar{\phi}^k$ vanishes identically in~$[0,r_0]$.

Therefore, for~$\sigma\le 0$, the eigenfunction~$\phi$ must be radial, and we deduce that,
\begin{align}
 0=(\lambda+\sigma)\int_{B_1} V_\lambda \oscil{\phi}\dx
 =\int_{B_1} -\Delta\phi \dx
 =\int_{\p B_1} -\frac{\p\phi}{\p r}\ds
 = -|\sph^{n-1}| \phi'(1),
\end{align}
which is the same as~$\phi'(1)=0$.

\

\section{On the multiplicity of eigenvalues}
\label{sect:uniqueness}

In this section we are going to prove Proposition~\ref{prop:uniqueness}. Let~$\phi\in H^1_0(B_1)$ be a radial eigenfunction of~$\sigma$, which satisfies~\eqref{eq:eigenfunction-radial}.
In particular, as above, integration by parts gives $\phi'(1)=0$.
If~$\average{\phi}=0$, then~$\phi$ satisfies a classical elliptic PDE with~$\phi|_{\p B_1}=0$ and~$\p_\nu \phi|_{\p B_1}=0$, hence~$\phi\equiv 0$, which is impossible.
Therefore we have $\average{\phi}\neq 0$ as far as $\phi$ is nontrivial.
As a consequence by~\eqref{eq:eigenfunction-radial} we see that on~$\p B_1$,
\begin{align}
 \phi''(1)= (\lambda+\sigma) V_\lambda(1)\average{\phi}\neq 0.
\end{align}

Now if there were two independent eigenfunctions~$\phi_1,\phi_2\in H^1_0(B_1)$ of the eigenvalue~$\sigma$, then we would have,
\begin{align}
 \phi_j(1)=0, & & \phi'_j(1)=0, & & j=1,2
\end{align}
and~$\phi''_1(1)\neq 0$,~$\phi''_2(1)\neq 0$.
Thus we could find a linear combination
\begin{align}
 \Phi\equiv \alpha \phi_1 + \beta\phi_2
\end{align}
for some~$\alpha,\beta\in \R$ such that
\begin{align}
 \Phi''(1)=0.
\end{align}
Note that~$\Phi\in H^1_0(B_1)$ is also a radial eigenfunction of~$\sigma$ with~$\Phi(1)=0$,~$\Phi'(1)=0$.
Then~$\average{\Phi}\neq 0$ as otherwise we would have $\Phi\equiv 0$, which contradicts that~$\phi_1$ and~$\phi_2$ are linearly independent.
But then we should have again~$\Phi''(1)\neq 0$, which is the desired contradiction.

\

\section{concluding remarks: a degenerate case}\label{sec6}

The assumption that~$V_\lambda>0$ on~$\overline{B_1}$ is crucial in treatment of the problem,
as it implies that~$\phi''(1)\neq 0$.
This is usually the case in many applications.
For instance, as far as $\al>0$, to cover the plasma problem (where $f(t)=t^p$, see \cite{BJ1,BJ2}), it is enough to assume
\begin{align}
 f'>0 \qquad \mbox{ in } (0,+\infty)
\end{align}
which implies, for each~$\lambda>0$,
\begin{align}
 V_\lambda(x)=f'(\alpha_\lambda+ \lambda\psi_\lambda)>0
 \mbox{ on } \overline{B_1}
\end{align}
where~$(\alpha_\lambda,\psi_\lambda)$ is a solution of~\eqref{eq:MFE} with $\alpha_{\lambda}>0$.
Another well-known example is the Liouville nonlinearity, where $f(t)=e^t$ and in fact in this case
we are allowed to peak any $\al\in \mathds{R}$ (\cite{B2,BJLY-ARMA(2018)}).


However, in one of our aiming applications (\cite{BJ1,BJ2}) we also need to consider the case where $V_\lambda|_{\p B_1}=0$, more exactly~$f(t)=t^p$ and~$\alpha_{\lambda}=0$ in the above example.
This is a rather delicate limiting case for the study of the stability of the solutions of the plasma problem.
Actually, as far as we just assume $V_\lambda|_{\p B_1}=0$, the arguments provided above for the simplicity of the radial eigenfunctions and for the finiteness of the singular points may be not conclusive in general.
However if we knew that
\begin{align}\label{eq:model assumption}
 \lim_{r\to 1^{-}} \frac{V_\lambda(r)}{(1-r)^\beta} = v_0,
\end{align}
for some~$\beta>0$ and~$v_0>0$, then most of the main properties proved above still hold and
we will sketch the idea of how this is done in the rest of this section.\\
Remark that, interestingly enough, this is exactly what happens for the
model plasma problem with~$f(t)=t^p$ for some~$p>1$ and~$\alpha=0$. Indeed,
in this case by the Hopf Lemma  the radial solution~$\psi_0$ satisfies
$\p_r \psi_0(1)\neq 0$. As a consequence \eqref{eq:model assumption} holds for $V_\lambda$ with $\beta=p-1>0$.\\

We adopt the convention that a function~$u:[0,1]\to \R$ satisfies a~$\beta(>0)$-vanishing condition at~$r=1$ if there exists $a>0$ such that
$$
\lim_{r\to 1^{-}} \frac{u(r)}{(1-r)^\beta} = a.
$$
Then we have
\begin{thm}\label{thmsec7}
Assume that: for some~$0\leq k\in\mathbb{N}$,
\begin{itemize}
 \item $V_\lambda\in C^{k,\gamma}([0,1])$, for some $\gamma \in (0,1)$,
 \item $V_\lambda$ satisfies a~$\beta$-vanishing condition for some~$\beta\in (k, k+1]$.
\end{itemize}
%
Let $\phi$ be a solution of \eqref{eq:eigenfunction-radial} with $\average{\phi}\neq 0$.
Then~$\phi$ satisfies a~$(\beta+2)$-vanishing condition at~$r=1$.
\end{thm}

\begin{proof}
 We first prove the assertion for $k=0$.
 Since $V_\lambda\in C^{0,\gamma}([0,1])$ then by standard elliptic estimates $\phi$ is of class $C^{2,\gamma}$ near $r=1$.
 Clearly~\eqref{eq:eigenfunction-radial} takes the form,
 \begin{align}\label{eq:eigenfunction-radial-sec7}
  \phi''(r)+\frac{n-1}{r}\phi'(r)
  +(\lambda+\sigma)V_\lambda(r)\phi(r)
  =(\lambda+\sigma)V_\lambda(r)\average{\phi},
 \end{align}
 and since $\phi(1)=\phi'(1)=0=V_\lm(1)$, we deduce that~$\phi''(1)=0$ as well.
 Therefore we have
 \begin{align}
  |\phi''(r)|\leq C_2 (1-r)^\gamma,
  & &
  |\phi'(r)|\leq C_1 (1-r)^{1+\gamma},
  & &
  |\phi(r)|\leq C_0 (1-r)^{2+\gamma}.
 \end{align}

Let us divide \eqref{eq:eigenfunction-radial-sec7} by $(1-r)^{\beta}$ and observe that
$$
\frac{1}{r}\frac{\phi^{(1)}(r)}{(1-r)^{\beta}}\leq C \frac{(1-r)^{1+\gamma}}{(1-r)^{\beta}}\leq C
(1-r)^{\gamma}\to 0, \;r\to 1^{-},
$$
$$
\frac{\phi(r)}{(1-r)^{\beta}}\leq C \frac{(1-r)^{2+\gamma}}{(1-r)^{\beta}}\leq C
(1-r)^{1+\gamma}\to 0, \;r\to 1^{-},
$$
whence passing to the limit we find that
$$
\lim\limits_{r\to 1^{-}}\frac{\phi''(r)}{(1-r)^{\beta}}=
\lim\limits_{r\to 1^{-}}(\lm+\sg)\average{\phi}\frac{V_{\lm}(r)}{(1-r)^{\beta}}=(\lm+\sg)\average{\phi}v_0,
$$
which proves the claim for $k=0$.

\

For~$k\ge 1$, observe that a~$C^k$ function~$u$ satisfies a~$\beta(\in (k,k+1])$-vanishing condition at~$r=1$ if and only if~$u^{(j)}$ satisfies a~$(\beta-j)$-vanishing condition for all~$0\le j\le k$.
This is an immediate consequence of L'Hospital's rule.
In particular, for~$1\le j\le k$,
\begin{align}
 \lim_{r\to 1^{-}}\frac{V^{(j)}_\lambda}{(1-r)^{\beta-j}}= v_j
\end{align}
with~$v_j=\beta(\beta-1)\cdots (\beta-j+1)v_0>0$.
Moreover, in our case, it suffices to prove that~$\phi^{(2+k)}$ satisfies a~$(\beta-k)$-vanishing condition.

 Taking the $k$-th derivative of \eqref{eq:eigenfunction-radial} yields an equation of the form,
\begin{align}\label{eq:eigenfunction-radial-sec7-2}
 \phi^{(2+k)}(r)+\sum\limits_{j=1}^{k+2}c_j(r)\phi^{(2+k-j)}(r)
 =(\lambda+\sigma)V_\lambda^{(k)}(r)\average{\phi},
\end{align}
where $c_j(r)$, $j=1,\cdots,k+2$ are smooth functions of $r$ near $r=1$.
Since $V^{(j)}_\lm(1)=0$, $j=1,\cdots,k$,
we have
$$
|\phi^{(2+k-j)}(r)|\leq C_0 (1-r)^{j+\gamma},\;j=0,1,\cdots,2+k.
$$
Let us divide \eqref{eq:eigenfunction-radial-sec7-2} by $(1-r)^{\beta-k}$ and,
recalling that $\beta-k\in (0,1]$, observe that,
$$
\sum\limits_{j=1}^{k+2}|c_j(r)|\left|\frac{\phi^{(2+k-j)}(r)}{(1-r)^{\beta-k}}\right|\leq
C\sum\limits_{j=1}^{k+2}\frac{(1-r)^{j+\gamma}}{(1-r)^{\beta-k}}\leq
C\sum\limits_{j=1}^{k+2}(1-r)^{j-1+\gamma}\to 0,\,r\to 1^{-},
$$
whence passing to the limit we find that
$$
\lim\limits_{r\to 1^{-}}\frac{\phi^{(2+k)}(r)}{(1-r)^{\beta-k}}=
\lim\limits_{r\to 1^{-}}(\lm+\sg)\average{\phi}\frac{V_{\lm}^{(k)}(r)}{(1-r)^{\beta-k}}=(\lm+\sg)\average{\phi}v_k,
$$
which proves the claim for $k\geq 1$.

\end{proof}

%
%
%
\bigskip

Theorem \ref{thmsec7} guarantees that the argument for simplicity of the space of radial eigenfunctions associated to a fixed~$\sigma$ works as well and in particular we deduce that Proposition~\ref{prop:uniqueness} holds under the assumptions about $V_{\lm}$ of Theorem \ref{thmsec7}.\\
For the other results, Lemma~\ref{lemma:negative eigenvalue has radial eigenfunction} holds since the argument in the proof does not require the positivity of $V_\lambda$ at the boundary.\\
As for the main results, Theorem~\ref{thm:radial nodal domains-first} and Theorem~\ref{thm:radial nodal domains-general}, we have to clarify first what we mean by a \emph{generalized nodal domain}, since in this case we cannot impose that $\phi''(1)>0$.
However this degeneracy only occurs at~$r=1$ since the solution~$\psi$ to~\eqref{eq:MFE} is positive in the interior of the domain, whence by definition $V_\lambda$ has the same property as well.
Thus we take the following definition:
\begin{dfn}
 Let~$\phi$ be a radial eigenfunction of~\eqref{eq:eigenfunction-radial} with~$\average{\phi}>0$ in~$B_1$.
 A \emph{singular point} of~$\phi$ is a point~$r_0\in [0,1]$ such that
 \begin{itemize}
  \item $\phi(r_0)=0$, $\phi'(r_0)=0$;
  \item $\phi$ satisfies some~$\beta$-vanishing condition for some~$0<\beta<+\infty$ at~$r=r_0$.
 \end{itemize}
\end{dfn}
If~$V_\lambda$ is positive up to the boundary, then we see that the above definition is equivalent to Definition~\ref{dfn:singular point-I}. On the other side, under the assumptions of Theorem \ref{thmsec7},
we can still use the concept of \emph{generalized nodal domain} as above and
the proof in Section~\ref{sect:proof of first thm} still works in this setting. Indeed Theorem \ref{thmsec7}
in particular guarantees that there is only a finite number of generalized nodal domains.
Therefore the main Theorems~\ref{thm:radial nodal domains-first} and \ref{thm:radial nodal domains-general} are valid as well.

\bigskip

\noindent {\bf Data Availability Statement.}
Data sharing not applicable to this article as no datasets were generated or analyzed during the current study.

\

\bibliographystyle{siam}
\bibliography{nodal-domains}

\end{document}